\documentclass[10pt,a4paper,oneside,reqno]{amsart}

\usepackage{amsmath}
\usepackage{amssymb}
\usepackage{mathrsfs}
\usepackage{amsthm}
\usepackage[english]{babel}
\usepackage[final]{graphicx}
\usepackage{latexsym}
\usepackage{hyperref}

\usepackage[pagewise]{lineno}

\hypersetup{
    colorlinks=true,
    unicode=false,
    linkcolor=red,
    citecolor=blue,  
    bookmarks=true
    }

\providecommand{\norm}[1]{\lVert#1\rVert}

\newtheorem{teo}{Theorem}[section]
\newtheorem{cor}[teo]{Corolary}
\newtheorem{lem}[teo]{Lemma}
\newtheorem{prop}[teo]{Proposition}
\newtheorem{obs}[teo]{Remark}
\newtheorem{df}[teo]{Definition}
\newtheorem{afi}[teo]{Claim}
\newtheorem{que}{Question}[section]
\theoremstyle{definition}

\title[Dynamical coherence]{Dynamical coherence of partially hyperbolic diffeomorphisms on nilmanifolds isotopic to Anosov}
\author[Luis Pedro Pi\~{n}eyr\'ua]{Luis Pedro Pi\~{n}eyr\'ua}
\thanks{L.P.P was partially supported by CSIC Ini. Mod. 2 Id 72 (2015) and CSIC group 618}
\address{CMAT \\ Facultad de Ciencias \\
Universidad de la Rep\'ublica \\ Uruguay} 
\email{lpineyrua@cmat.edu.uy}

\begin{document}

\maketitle

\theoremstyle{plain}

\begin{abstract}
The purpose of this article is to obtain dynamically coherence of partially hyperbolic diffeomorphisms in certain classes of Anosov diffeomorphisms on nilmanifolds, extending a result due to T. Fisher, R. Potrie and M. Sambarino \cite{FPS} on the torus.
\end{abstract}

\section{Introduction}
The main purpose of the theory of dynamical systems is to understand the temporary evolution of a given system. Another way of saying is, given a fixed space and a law that rules the motion on it, try to predict the asymptotic behaviour for the most quantity of possible trajectories. 

In some cases this behaviour turns out trivial because the simplicity of the dynamics, but in other cases the evolution law presents special futures that makes the dynamics unpredictable or chaotic. 
The paradigmatic examples of chaotic systems are called \textit{Anosov diffeomorphisms}: a diffeomorphism $f:M \to M$ is \textit{Anosov} or \textit{globally hyperbolic}, if there is a splitting of the tangent bundle in a direct sum of sub bundles which are $Df$-invariant $TM=E^{s}\oplus E^{u}$ such that $Df$ contracts vectors of $E^{s}$ and $E^{u}$ exponentially in the future and the past respectively (precise definitions are given in section \ref{spre}). 

The property of expansion-contraction called \textit{hiperbolicity} induces some interesting dynamical properties such as expansivity, transitivity, the existence of periodic points of arbitrary large periods and positive metric entropy to name a few. The concept of hyerbolic set was first introduced by S. Smale in \cite{Sm} where he makes a detailed study of hyperbolicity and propose some guiding problems. The most relevant perhaps is the problem of classifying Anosov diffeomorphisms up to conjugacy. We say that two diffeomorphisms $f:M \to M$ and $g:N\to N$ are topologically equivalent or  \textit{conjugated} if there exist a homeomorphism $h:M \to N$ such that $h \circ f=g \circ h$. 

In the late seventies there was a very good knowledge about how this classification was for Anosov sytems due to the works of J. Franks, S. Newhouse and  A. Manning. If we put toghether the works \cite{Fr}, \cite{Ne} we obtain that if $M$ is a connected, compact riemannian manifold of dimension $n$ without boundary and $f:M \to M$ is an Anosov diffeomorphism of codimendion 1, then $M=\mathbb{T}^{n}$ and $f$ is conjugated to a linear Anosov diffeomorphism. On the other hand, in \cite{Fr1} and \cite{Man} the authors proved that if $f:M\to M$ is an Anosov diffeomorphism on a nilmanifold $M$, then $f$ is topologically conjugated to an Anosov automorphism. 

Despite this important results, some questions are left to be answered. For example, is still an open problem to decide which manifolds support Anosov diffeomorphism, and if this diffeomorhpisms are always transitive. 

In the attempt to generalize these results the definition of Anosov is weakened giving place to  \textit{partially hyperbolic diffeomorphisms}. We say that a diffeomorphism $f:M \to M$ is partially hyperbolic if the tangent bundle splits in a direct sum of three $Df$-invariant sub bundles  $TM=E^{ss}\oplus E^{c}\oplus E^{uu}$ such that the sub bundles $E^{ss}$ and $E^{uu}$ contract vectors exponentially in the future and the past respectively, and the center bundle $E^{c}$ has an intermediate behaviour. In this way partially hyperbolic diffeomorphisms are a generalization of Anosov diffeomorphisms (with trivial $E^{c}$ bundle).

As in the Anosov case, we are interested in classifying the partially hyperbolic diffeomorphisms. A key tool in the Franks-Newhouse-Manning classification, is the existence of invariant foliations tangent to the stable/unstable distributions. These foliations always exist for the stable/unstable bundles, both for Anosov and partially hyperbolic diffeomorphisms (\cite{HPS}). However, the central bundle $E^{c}$ is not always integrable, which means, there is no invariant foliation tangent to $E^{c}$ in every point. There are examples of partially hyperbolic diffeomorhpims whose central bundle is not integrable. The first example of this type was a partially hyperbolic dffeomorphisms on a nilmanifold of dimension 6. This example appeared for the first time in \cite{Sm} as an Anosov diffeomorphism in a manifold which is not a torus. Years later A. Wilkinson \cite{W} observed that rearrenging the bundles one can obtain a partially hyperbolic diffeomorphism whose central bundle $E^{c}$ is not integrable because the Frobenius condition fails (section \ref{sexamples}). In \cite{RHRHU} there's a second example of this type in the torus $\mathbb{T}^{3}$. 

We say that a partially hyperbolic diffeomorphism is \textit{dynamically coherent} (DC) if there are invariant foliations, tangent to the central-stable, central-unstable distributions at every point (and in that case, there is a central foliation too).

The first result about dynamically coherence is due to M. Brin \cite{Br} where he proves that an absolute partially hyperbolic diffeomorphism$\footnote{absolut partially hyperbolic is a strong version of partial hyperbolicity.}$ is dynamically coherent if the stable/unstable leaves are quasi isometric in the universal cover. Then, in \cite{BBI} dynamically coherence is obtained for absolute partially hyperbolic diffeomorphisms in the torus $\mathbb{T}^{3}$ by using Brin's criterion. 
In the last years, R. Potrie and A. Hammerlindl in a series of works \cite{Po}, \cite{HamPo}, \cite{HamPo2} proved dynamical coherence for partially hyperbolic diffeomorpfisms in 3 dimensional manifolds with solvable fundamental group modulus a topological obstruction.

Despite these results, when the dimension of the central distribution is greater than one very few is well known. The first result in this direction is due to T. Fisher, R. Potrie and M. Sambarino \cite{FPS} where they obtain dynamically coherence for partially hyperbolic diffeomorphisms isotopic to linear Anosov on tori $\mathbb{T}^{n}$, provided that the whole isotopy path is inside the space of partially hyperbolic diffeomorphisms. In that paper there are no restrictions about the central dimension, and dynamically coherence is obtained for large subsets of partially hyperbolic diffeomorphisms (connected components of linear Anosov diffeomorphisms). They also mention that it would be possible to applied their techniques to the nilmanifold case but this has to be done with some care. The purpose of this paper is answer this in an affirmative way. Next we present the context in which we will work. 

Let $A:M\to M$ be an Anosov automorphism on a nilmanifold $M=G/\Gamma$, where $G$ is a connected, simply connected nilpotent Lie group, and $\Gamma \subset G$ a discrete and cocompact subgroup. The tangent space $T_{e}M$ admits Lie algebra structure and the differential $DA:TM \to TM$ induces a splitting on the tangent bundle of the form $TM=E^{ss}_{A}\oplus E^{ws}_{A}\oplus E^{wu}_{A}\oplus E^{uu}_{A}$ (see section \ref{spre}). There may be many possibilities for the dimension of these bundles. We will suppose through all this work that the \textit{central bundle} $E^{c}_{A}=E^{ws}_{A}\oplus E^{wu}_{A}$ is a Lie subalgebra of $T_{e}M$. This is always the case when the manifold is the torus $\mathbb{T}^{n}=\mathbb{R}^{n}/\mathbb{Z}^{n}$ because $\mathbb{R}^{n}$ is an abelian Lie algebra, and for that reason any linear subspace will be a Lie subalgebra. This represents the only difference from the original case in the torus $\mathbb{T}^{n}$ \cite{FPS} to the general nilmanifold case where the central bundle is not always closed under the Lie bracket operation (section \ref{sexamples}).

We denote $\textnormal{PH}(M)=\{f:M \to M \ \textnormal{partially hyperbolic}\}$. Now given $A$ as above, we are going to consider 
\begin{equation*}
\textnormal{PH}_{A}(M)= \left\{
f \in \textnormal{PH}(M): f \simeq A, \ \textnormal{dim}E^{ss}_{f}=\textnormal{dim}E^{ss}_{A}, \ \textnormal{dim}E^{uu}_{f}=\textnormal{dim}E^{uu}_{A} 
\right\}
\end{equation*} where $f\simeq A$ means the maps are isotopic. Given $f\in \textnormal{PH}_{A}(M)$ we know from \cite{Fr} that there exist a continuous and surjective map $H_{f}:G \to G$ such that $A \circ H_{f}=H_{f}\circ \tilde{f}$, where $\tilde{f}$ denotes the lift of $f$ to the universal cover $G$.
We say that a dynamical coherent diffeomorphism $f \in \textnormal{PH}_{A}(M)$ is \textit{center fibered} (CF) if $H_{f}^{-1}(\tilde{\mathcal{W}}_{A}^{c}(H_{f}(x)))=\tilde{\mathcal{W}}^{c}_{f}(x)$. This means that different center leaves $f$ are sent by $H_{f}$ surjectively to different center leaves of $A$. We are going to note
\begin{equation*}
\textnormal{PH}_{A}^{0}(M)=\left\{ 
\begin{array}{cc}
& \textnormal{connected componentes of PH}_{A}(M) \ \textnormal{which contains a} \ \ \\
& \textnormal{DC and CF partially hyperbolic diffeomorphism} \ \ 
\end{array}
\right\}
\end{equation*}
We remark that the algebraic Anosov $A$ itself is center fibered, so the set $\textnormal{PH}_{A}^{0}(M)$ is a non-empty open set with at least one connected component.
Now we are ready to state the main result of this work:
\begin{teo}\label{teoA}
Every $f\in \textnormal{PH}_{A}^{0}(M)$ is dynamically coherent and center fibered.
\end{teo}
For the proof of this theorem it is crucial the hypothesis we have made: the center bundle $E^{c}_{A}$ is a Lie subalgebra. It remains open the question about if the reciprocal is also true:
\begin{que}
Given $f\in \textnormal{PH}(M)$ dynamically coherent and isotopic to an Anosov automorphism $A$. Is the central bundle $E^{c}_{A}$ a Lie subalgebra?
\end{que} 
A positive answer to this question would closed the problem about dynamical coherence in linear Anosov isotopy classes. It will be sufficient to check the behaviour of the Lie bracket of the linear part of $f$ (a purely algebraic condition) to establish dynamically coherence for $f$.  

\textbf{Organization of the paper:} The article is divided as follows. In section 2 we introduce the definitions and necessary preliminaries. In section 3 we present the Borel-Smale-Wilkinson example. In section 4 we deal with global product structure for invariant manifolds. Section 5 is devoted to a dynamical coherence criterion and finally in section 6 we prove the main theorem.

\section{Preliminaries} \label{spre} 

Let $f:M \to M$ be a diffeomorphism in a connected, compact an boundaryless manifolds $M$. We say that $f$ is an \textit{Anosov diffeomorphism} if the following properties hold:  

\begin{itemize}
\item[a)]There exist a splitting of the tangent bundle $TM=E^{s}\oplus E^{u}$ in two $Df$-invariant and continuous subbundles:
\begin{itemize}
\item[i)] $Df_{x}(v^{s})\in E^{s}(f(x)), \ \text{for every} \ v^{s} \in E^{s}(x)$. 
\item[ii)] $Df_{x}(v^{u})\in E^{u}(f(x)), \ \text{for every} \ v^{u} \in E^{u}(x)$.
\end{itemize}
\item[b)] There exist a riemannian metric $\norm{\cdot}$ and constants $C>0$, $\lambda \in (0,1)$ such that:
\begin{itemize}
\item[i)] $\norm{Df^{n}_{x}(v)}\leq C \lambda^{n}\norm{v}$ for every $v \in E^{s}(x)$ and $n >0$.
\item[ii)] $\norm{Df^{-n}_{x}(v)} \leq C \lambda^{n}\norm{v}$ for every $v \in E^{u}(x)$ and $n>0$.
\end{itemize}
\end{itemize} We call $E^{s}$ and $E^{u}$ the stable and unstable subbundles respectively. 

If we add an extra sub bundle to the Anosov definition we obtain what is called partially hyperbolic diffeomorphism: let $M$ be as above, a diffeomorphism $f:M \to M$ is \textit{partially hyperbolic} if the following conditions hold: 

\begin{itemize}
\item[a)] There exist a splitting of the tangent bundle in three $Df-$invariant and continuous subbundles: $TM=E^{ss}_{f}\oplus E^{c}_{f}\oplus E^{uu}_{f}$.
\item[b)] There exist a riemannian metric $\norm{\cdot}$ and constants $C>0$, $\lambda \in (0,1)$ such that for very $x \in M$ and unitary vectors $v^{\sigma}\in E^{\sigma}_{f}(x)$, $\sigma=ss,c,uu$: 
\begin{itemize}
\item[i)] $\lambda^{-1}\norm{Df_{x}v^{ss}}<\norm{Df_{x}v^{c}}<\lambda\norm{Df_{x}v^{uu}}$.
\item[ii)] $\norm{Df^{n}_{x}v^{ss}}\leq C\lambda^{n}$, $\norm{Df^{-n}_{x}v^{uu}}\leq C \lambda^{n}$, $\forall n \geq 0$.
\end{itemize}
\end{itemize}

Until the date the most general examples of Anosov diffeomorphisms are constructed as follows. 
Let $G$ be a connected, simply connected Lie group of dimension $n$ and $A:G \to G$ a Lie group isomorphism. Since the neutral $e$ of $G$ is fixed by $A$, the differential $DA_{e}:T_{e}G \to T_{e}G$ is a linear isomorphism and induces a Lie algebra isomorphism between the corresponding Lie algebras $dA:\mathfrak{g} \to \mathfrak{g}$. This correspondence between $\mathfrak{g}$ and $T_{e}G$ comes from the linear isomorphism $\alpha:\mathfrak{g}\to T_{e}G$ which sends $X\in \mathfrak{g}$ to the vector $X(e)\in T_{e}G$ and it also conjugates the maps $DA_{e}$ and $dA$: $$DA_{e}\circ \alpha = \alpha \circ dA$$ We say that the automorphism $A$ is \textit{Anosov} if the linear transformation $dA: \mathfrak{g} \to \mathfrak{g}$ is hyperbolic, i.e. it has no eigenvalues of modulus equal to one. 
In that the case we can decompose $\mathfrak{g}$ in a direct sum of eigenspaces $\mathfrak{g}^{s}$ and $\mathfrak{g}^{u}$, where $\mathfrak{g}^{s}$ is the sum of the eigenspaces associated to the eigenvalues of modulo smaller than 1, and $\mathfrak{g}^{u}$ is the sum of the eigenspaces associated to the eigenvalues of modulo larger than 1.
Since $dA$ is hyperbolic we obtain that $\mathfrak{g}=\mathfrak{g}^{s}\oplus \mathfrak{g}^{u}$. For this decomposition of $\mathfrak{g}$ there exist an inner product $\langle,\rangle$ in $\mathfrak{g}$ such that its corresponding norm contracts vectors in $\mathfrak{g}^{s}$ and $\mathfrak{g}^{u}$ exponentially for the future and the past respectively. This means there are constants $C > 0$ and $0 < \lambda < 1$ such that:

\begin{itemize}
\item $\norm{dA^{n}(v)} \leq C \lambda^{n}\norm{v}$ for all $v \in \mathfrak{g}^{s}$ and $n >0$.
\item $\norm{dA^{-n}(v)} \leq C \lambda^{n}\norm{v}$ for all $v \in \mathfrak{g}^{u}$ and $n>0$.
\end{itemize}
The hyperbolicity of $dA$ implies that the eigenspaces $\mathfrak{g}^{s}$ and $\mathfrak{g}^{u}$ are Lie subalgebras, i.e. $\mathfrak{g}^{s}$ and $\mathfrak{g}^{u}$ are closed under Lie bracket operation.

\begin{prop} \label{propsubalg}
Let $A:G \to G$ be an Anosov automorphism with a decomposition of the Lie algebra of the form: 
$\mathfrak{g}=\mathfrak{g}^{s} \oplus \mathfrak{g}^{u}$. Then, the eigenspaces $\mathfrak{g}^{s}$ and $\mathfrak{g}^{u}$ are Lie subalgebras.
\begin{proof}
We will prove the stable case $\mathfrak{g}^{s}$. The proof for the unstable case is completely analogous. 
Given $X,Y \in \mathfrak{g}^{s}$, we have to prove that $[X,Y] \in \mathfrak{g}^{s}$.

We know that $\norm{dA^{n}X}$ and $\norm{dA^{n}(Y)}$ goes to 0 when $n \to +\infty$ because $X,Y \in \mathfrak{g}^{s}$. The Lie bracket is a bilinear operation, so we know it is a continuous application. Then $[dA^{n}X,dA^{n}Y]\to 0$ when $n \to +\infty$. Since $dA:\mathfrak{g}\to \mathfrak{g}$ is a Lie algebra homomorphism it preserves the Lie bracket and we can conclude that $dA^{n}([X,Y])=[dA^{n}X,dA^{n}Y] \to 0$ whenever $n \to +\infty$. This proves that $[X,Y]\in \mathfrak{g}^{s}$.
\end{proof}
\end{prop}
The isomorphism $\alpha :\mathfrak{g} \to T_{e}G$ allows us to send the inner product of $\mathfrak{g}$ to the tangent space $T_{e}G$:
$$\langle v,w \rangle_{e}=\langle \alpha^{-1}(v),\alpha^{-1}(w)\rangle$$
Then translating the inner product in $T_{e}M$ by left multiplication we obtain an inner product in every point $x\in G$:
$$\langle v,w\rangle_{x}=\langle D(L_{x})_{e}^{-1}(v),D(L_{x})_{e}^{-1}(w)\rangle_{e}$$ 
It's easy to see that this defines a Riemmanian metric, which is invariant under left translations (see for example \cite{DoC}). Now we define the stable distribution $E^{s}(x) \subset T_{x}G$ by $E^{s}(e)=\alpha(\mathfrak{g}^{s})$ and then translating by left multiplication: $E^{s}(x)=D(L_{x})_{e}(E^{s}(e))$. In a similar way we define  $E^{u}(x)$. Let's see that $A:G\to G$ with this splitting and this Riemannian metric is an Anosov diffeomorphism.  

\subsection*{Invariance by the differential map}
Take $v \in E^{\sigma}(x)$, $\sigma=s,u$. By definition we have, 
$$v \in E^{\sigma}(x)= D(L_{x})_{e}(E^{\sigma})=D(L_{x})_{e}(\alpha(\mathfrak{g}^{\sigma}))$$
Then there exist a vector $w^{\sigma}\in \mathfrak{g}^{\sigma}$ such that $v=D(L_{x})_{e}(\alpha(w^{\sigma}))$. It follows that
$$ DA_{x}(v) = DA_{x}(D(L_{x})_{e}(\alpha(w^{\sigma})))=D(A \circ L_{x})_{e}(\alpha (w^{\sigma}))$$
Since $A$ is a group homomorphism, we have $A \circ L_{x}=L_{f(x)}\circ A$ and then 
\begin{eqnarray*}
DA_{x}(v)&=&D(L_{A(x)}\circ A)_{e}(\alpha(w^{\sigma}))= D(L_{A(x)})_{e} \circ (DA_{e}\circ \ \alpha)(\mathfrak{g}^{\sigma})\\
&=&D(L_{A(x)})_{e}\circ (\alpha \ \circ \ dA)(w^{\sigma})
\end{eqnarray*}
On the other hand, $\mathfrak{g}^{\sigma}$ is $dA$-invariant, so we have that $\alpha(dA (w^{\sigma}))\in \alpha (\mathfrak{g}^{\sigma})=E^{\sigma}(e)$.
Finally, 
$$DA_{x}(v)\in D(L_{A(x)})_{e}(E^{\sigma}(e))=E^{\sigma}(A(x))$$ as we wanted to see.

\subsection*{Contraction and expansion} 
We are going to prove the $E^{s}$ case. The proof for $E^{u}$ is analogous. Take $v \in E^{s}(x)$. By definition $D(L_{x})^{-1}_{x}(v)\in E^{s}(e)=\alpha(\mathfrak{g}^{s})$. Then we have that there is a vector $w^{s}\in \mathfrak{g}^{s}$ such that $\alpha(w^{s})=D(L_{x^{-1}})_{x}(v)$.
Then $DA^{n}_{x}(v)\in E^{s}(A^{n}(x))$ and therefore
$$ \norm{DA^{n}_{x}(v)}_{A^{n}(x)}=\norm{D(L_{A^{n}(x)})^{-1}_{A^{n}(x)}(DA^{n}_{x}(v))}_{e}=\norm{D(L_{A^{n}(x)}^{-1}\circ A^{n})_{x}(v)}_{e}
$$
Since $A$ is a group homomorphism we get  
$$L^{-1}_{A^{n}(x)}\circ A^{n}=L_{A^{n}(x)^{-1}}\circ A^{n}=L_{A^{n}(x^{-1})} \circ A^{n}=A^{n}\circ L_{x^{-1}}$$
Then
\begin{eqnarray*}
\norm{DA^{n}_{x}(v)}_{A^{n}(x)}&=&\norm{D(A^{n}\circ L_{x^{-1}})(v)}_{e}=\norm{DA^{n}_{e}\circ D(L_{x^{-1}})_{x}(v)}_{e} \\
&=& \norm{DA^{n}_{e}\circ \alpha(w^{s})}_{e}
\end{eqnarray*}
Finally since $\alpha$ conjugates the maps $DA_{e}$ and $dA$ we obtain:
\begin{eqnarray*}
\norm{DA^{n}_{x}(v)}_{A^{n}(x)} &=& \norm{DA^{n}_{e}\circ \alpha(w^{s})}_{e} =\norm{\alpha \circ dA^{n}(w^{s})}_{e}
= \norm{dA^{n}(w^{s})} \\ 
&\leq & C\lambda^{n}\norm{w^{s}}
= C\lambda^{n}\norm{\alpha(w^{s})}_{e}=C\lambda^{n}\norm{D(L_{x^{-1}})_{x}(v)}_{e} \\
&=& C\lambda^{n}\norm{v}_{x} \qedhere
\end{eqnarray*} 
\begin{flushright}
$\Box$
\end{flushright}
This proves that $A:G \to G$ is an Anosov diffeomorphism in a simply connected manifold. The next result due to N. Jacobson shows the first algebraic/topological obstruction to these constructions. 

\begin{teo} [\cite{J}] \label{Jaco} Let $\phi :\mathfrak{g}\to \mathfrak{g}$ be a Lie algebra automorphism which is hyperbolic as a linear transformation. 
Then $\mathfrak{g}$ is nilpotent. 
\end{teo}
This theorem says that $\mathfrak{g}$ must be a nilpotent Lie algebra, and this occurs if and only if the group $G$ is nilpotent (\cite{SW}). For that reason we are going to work exclusively with nilpotent Lie groups from now on .  

To obtain an example on a compact manifold let's suppose that $G$ has a discrete and cocompact subgroup $\Gamma$ which is $A$-invariant: $A(\Gamma)=\Gamma$. Then $A$ induces a map $f_{A}$ in the corresponding quotient space $M=G/\Gamma$, given by 
$$f_{A}:M \to M \ , \ \ f_{A}(x \cdot \Gamma)=A(x) \cdot \Gamma$$
It's easy to see that $f_{A}$ is well defined, because $A$ is a homomorphism and $\Gamma$ is $A$-invariant. It is also invertible because $f^{-1}_{A}=f_{A^{-1}}$. The invariance of the decomposition and the metric by left translations (and in particular of $\Gamma$) implies that we can send the riemmanian metric of $G$ to the quotient $M=G/\Gamma$ obtaining an Anosov diffeomorphism $f_{A}:G/\Gamma \to G/\Gamma$. We call $M=G/\Gamma$ a \textit{nilmanifold}, that is a riemannian manifold which is the quotient of a connected simply connected nilpotent Lie group by the action of a discrete and cocompact subgroup.
The diffeomorphisms $f_{A}$ obtained this way are called $\textit{Anosov automorphisms}$ or $\textit{hyperbolic automorphisms}$. In section \ref{sexamples} we'll see a few examples.

In some cases, when the dimension of $M$ is bigger than three, we can decompose the stable and unstable subalgebras in the form $\mathfrak{g}^{s}=\mathfrak{g}^{ss}\oplus \mathfrak{g}^{ws}$ and $\mathfrak{g}^{u}=\mathfrak{g}^{wu}\oplus \mathfrak{g}^{uu}$ obtaining a dominated splitting of the Lie algebra of the form: 
$$\mathfrak{g}=\mathfrak{g}^{ss}\oplus \mathfrak{g}^{ws}\oplus \mathfrak{g}^{wu}\oplus \mathfrak{g}^{uu}$$
We are going to call $\mathfrak{g}^{ss}$, $\mathfrak{g}^{ws}$, $\mathfrak{g}^{wu}$ and $\mathfrak{g}^{uu}$ the \textit{strong stable}, \textit{weak stable}, \textit{weak unstable} and \textit{strong unstable} subspaces respectively. The direct sums $\mathfrak{g}^{c}=\mathfrak{g}^{ws}\oplus \mathfrak{g}^{wu}$, $\mathfrak{g}^{cs}=\mathfrak{g}^{ss}\oplus \mathfrak{g}^{c}$ and $\mathfrak{g}^{cu}=\mathfrak{g}^{c}\oplus \mathfrak{g}^{uu}$ are the \textit{central}, \textit{central stable} and \textit{central unstable} subspaces. In proposition \ref{propsubalg} we saw that $\mathfrak{g}^{s}$ and $\mathfrak{g}^{u}$ are Lie subalgebras, however, we are going to make the following assumption.
 
\begin{obs} \label{obsliesub}
We will assume through all the work that $\mathfrak{g}^{cs}$ and $\mathfrak{g}^{cu}$ are Lie subalgebras. Hence, $\mathfrak{g}^{c}$ is also a Lie subalgebra for being an intersection of Lie subalgebras.
\end{obs}

Once again we can make the same construction and obtain the corresponding distributions $E^{\sigma}(x)=D(L_{x})_{e}(\alpha(\mathfrak{g}^{\sigma}))$ for $\sigma=ss,ws,wu,uu$ and $c$.
With this decomposition we get an Anosov diffeomorphism with a hyperbolic splitting of the form: 
$$TM=E^{ss}\oplus E^{ws}\oplus E^{wu}\oplus E^{uu}$$
with $E^{ss}\oplus E^{ws}$ and  $E^{wu}\oplus E^{uu}$ the stable and unstable distributions. Note that if we call $E^{c}=E^{ws}\oplus E^{wu}$, we get a partially hyperbolic splitting: 
$$TM=E^{ss}\oplus E^{c}\oplus E^{uu}$$
This shows that in this case we can think an Anosov diffeomorphism as a partially hyperbolic one. We will back to this in section \ref{sexamples}. Now we know the group $G$ must be nilpotent, the next result becomes useful.

\begin{teo}\cite{Mal} \label{teomalcev}
\
\begin{enumerate} 
\item A necessary a sufficient condition for a discrete group $\Gamma$ to occur as a cocompact subgroup of a simply connected nilpotent Lie group si that $\Gamma$ be a finitely generated nilpotent group containing no elements of finite order.
\item A necessary and sufficient condition on a nilpotent simply connected Lie group $G$ that there exist a discrete cocompact subgroup $\Gamma$ is that the Lie algebra of $G$ has rational constants of structure in some basis. 
\item \label{mal3} If $\Gamma_{i}$ is a discrete and cocompact subgroup of a simply connected nilpotent Lie group  $G_{i}$, $i=1,2$, then any isomorphism $\Gamma_{1}\to \Gamma_{2}$ can be uniquely extended to an isomorphism $G_{1}\to G_{2}$. 
\end{enumerate}
\end{teo} 

We have constructed Anosov and partially hyperbolic diffomorphism in an algebraic way. We can do the other way around. Given a diffeomorphism $f:G/\Gamma \to G/\Gamma$ we know from \cite{Hat} that $f$ induces an automorphism in its fundamental group $f_{*}:\Gamma \to \Gamma$. As $\Gamma$ is a discrete and cocompact subgroup and $G$ is nilpotent, we have from the previous Theorem \ref{teomalcev} part  \ref{mal3} that there is a unique isomorphism $A:G\to G$ such that $A|_{\Gamma}=f_{*}$. We call $A$ the \textit{linear part} of $f$. From \cite{Fr} we know that a diffeomorphism $f:M \to M$ is semiconjugated to its linear part $A$, as long as $A$ is hyperbolic. Whenever $f \in \textnormal{PH}_{A}(M)$ it's clear that the linear part of $f$ is $A$. 

As we mentioned in the introduction, one of the most important properties of a hyperbolic spitting is the existence of invariant foliations tangent to each one of the distributions. The solution to the existence of these foliations was given by Frobenius in the $C^{1}$ case, see \cite{Wa} for example. The problem with Anosov and partially hyperbolic diffeomorphisms is that the distributions are only Holder continuous even if the diffeomorphism is $C^{r}$ (see \cite{A}). The answer to this problem is given by the stable manifold theorem. The idea of the proof is due to Haddamard and is called the \textit{graph transformation argument}, see \cite{HPS}.

\section{The Borel-Smale-Wilkinson example} \label{sexamples}

The example we are going to present appeared for the first time in \cite{Sm} and it is attributed by S. Smale to A. Borel. The example orginally was presented as an Anosov diffeomorphism in a compact orientable manifold that is not a torus. Years later A. Wilkinson \cite{W} observed that putting together weak sub bundles, one creates a partially hyperbolic diffeomorphism whose central distribution is not integrable. For a more detailed presentation of these examples see \cite{Sm}, \cite{BuW} or \cite{Ham}. 
We now give a brief description of these examples.

Take $\mathcal{H}$ the Heisenberg group, that is the subgroup of matrices in $\textnormal{SL}(3,\mathbb{R})$ of the form
$$ \left(
\begin{array}{ccc}
1 & x & z \\
0 & 1 & y \\
0 & 0 & 1 
\end{array}\right)
$$
with $x,y,z \in \mathbb{R}$. Identifying $(x,y,z)$ with the upper triangular matrix, the product in $\mathcal{H}$ has the form: $$(x,y,z)\cdot(x',y',z')=(x+x',y+y',z+z'+xy')$$
Now we have that $\mathcal{H}$ is a connected, simply connected, nilpotent Lie group diffeomorphic to $\mathbb{R}^{3}$ and clearly non abelian. Its corresponding Lie subalgebra $\mathfrak{h}$ is generated by the matrices 
$$X=\left(
\begin{array}{ccc}
0 & 1 & 0 \\
0 & 0 & 0 \\
0 & 0 & 0 
\end{array} \right), \ 
Y=\left(
\begin{array}{ccc}
0 & 0 & 0 \\
0 & 0 & 1 \\
0 & 0 & 0 
\end{array} \right), \ 
Z=\left(
\begin{array}{ccc}
0 & 0 & 1 \\
0 & 0 & 0 \\
0 & 0 & 0
\end{array}\right)
$$
These matrices satisfy the following relations: $[X,Z]=[Y,Z]=0, [X,Y]=Z$.
If we identify $(a,b,c)$ with $aX+bY+cZ \in \mathfrak{h}$ the exponential map exp:$\mathfrak{h}\to \mathcal{H}$ is a diffeomorphism and its formula is given by 
$$\textnormal{exp}(a,b,c)=\left(\begin{array}{ccc}
1 & a & c+\frac{1}{2}ab \\
0 & 1 & b \\
0 & 0 & 1
\end{array}\right)$$
Now consider the group $\mathcal{G}=\mathcal{H}\times \mathcal{H}$ with the direct product group structure. We get that $\mathcal{G}$ is a connected, simply connected nilpotent Lie group diffeomorphic to $\mathbb{R}^{6}$. 
Its Lie algebra $\mathfrak{g}=\mathfrak{h}\oplus \mathfrak{h}$ is generated by $\{X_{1},Y_{1},Z_{1},X_{2},Y_{2},Z_{2}\}$. Note that the only non-trivial relations are  
$$[X_{1},Y_{1}]=Z_{1}, \ [X_{2},Y_{2}]=Z_{2}
$$
Now identify $(c,b,a,a',b',c') \in \mathbb{R}^{3}\times \mathbb{R}^{3}$ with $aX_{1}+bY_{1}+cZ_{1}+a'X_{2}+b'Y_{2}+c'Z_{2} \in \mathfrak{g}$.
Take a matrix $A \in \textnormal{SL}(2,\mathbb{Z})$ and suppose that $\lambda >1$ and $\lambda^{-1}<1$ are their eigenvalues. Now $\lambda$ and $\lambda^{-1}$ are units in the ring of integers. The field $\mathbb{Q}(\lambda)$ is a quadratic extension of $\mathbb{Q}$; it's Galois involution $\sigma$ interchanges $\lambda$ and $\lambda^{-1}$. Now if we take $\tilde{\Gamma}\subset \mathfrak{g}$ as the set of vectors of the form:
$$\left( \frac{1}{2}w,v,u,\sigma(u),\sigma(v),\sigma\left(\frac{1}{2}w\right)\right)
$$ with $u,v,w \in \mathbb{Z}[\lambda]$ the ring of algebraic integers in $\mathbb{Q}(\lambda)$. It can be proved that $\tilde{\Gamma}$ is an irreducible and cocompact lattice of $\mathfrak{g}$. Then it's easy to see that $\Gamma=\textnormal{exp}(\tilde{\Gamma})$ is a discrete and cocompact subgroup of $\mathcal{G}$. 
Now for any pair of real numbers $\alpha$ and $\beta$, the linear map $B$ $$B: (c,b,a,a',b',c')\mapsto (c\lambda^{\alpha + \beta},b\lambda^{\beta},a\lambda^{\alpha},a'\lambda^{-\alpha},b'\lambda^{-\beta},c'\lambda^{-\alpha-\beta}) 
$$
is an automorphism of $\mathfrak{g}$ and induces an homomorphism $F_{B}:\mathcal{G}\to \mathcal{G}$ whose derivative at the identity is $B$. If $\alpha, \beta \in \mathbb{Z}$ the automorphism $B$ preserves $\tilde{\Gamma}$ and we obtain a diffeomorphism $f_{B}:\mathcal{G}/\Gamma \to \mathcal{G}/\Gamma$. If one of $\alpha, \beta, \alpha + \beta$ is non zero, then $f_{B}$ is partially hyperbolic  and if all three are non zero, $f_{B}$ is Anosov. Assume that $\alpha+\beta > \beta \geq \alpha >0$. In this case $f_{B}$ is Anosov: the central bundle is trivial, the stable bundle $E^{s}$ is generated by $X_{2}, Y_{2}, Z_{2}$ and the unstable bundle $E^{u}$ by $X_{1}, Y_{1},Z_{1}$. This way we get an Anosov diffeomorphism $f_{B}:\mathcal{G}/\Gamma \to \mathcal{G}/\Gamma$ in a six dimensional nilmanifold that is not a torus (Lie algebra non abelian). 

This is the form in which this example originally appeared in \cite{Sm}, but as we mentioned at the beginning of this section, there are several ways in which one can think about this example. These are the following: 

\begin{itemize}
\item In \cite{W} A. Wilkinson made the following observation: take the stable bundle $E^{s}$ generated  by $Z_{2}$, the unstable bundle $E^{u}$ is generated by $Z_{1}$ and the central bundle $E^{c}$ generated by the remaining fields $X_{1}, Y_{1}, X_{2}$ and  $Y_{2}$. With this splitting $f_{B}$ is a partially hyperbolic diffeomorphism. The interesting thing about this example is that the central bundle $E^{c}$ is not integrable because is not closed under the Lie bracket operation: $[X_{1},Y_{1}]=Z_{1} \in E^{u}$. This implies that $f_{B}$ is not dynamically coherent.

\item A third way of seing this is due to A. Hammerlindl. One chooses the bundle $E^{u}$ to be generated by $Z_{1}$, $Y_{1}$ and $X_{1}$, the central bundle $E^{c}$ generated by $X_{2}$ and $Y_{2}$ and the stable bundle $E^{s}$ generated by $Z_{2}$. Once again $f_{B}$ is not dynamically coherent and moreover, if we take $\alpha=\beta$ we obtain center-bunching (see \cite{BuW} and references therein).

\item Now we are going to see it in fourth way, a much simpler one: the unstable bundle $E^{u}$ is generated by $Z_{1}$, $Y_{1}$, the central bundle $E^{c}$ generated by $X_{1}$, $X_{2}$ and the stable bundle $E^{s}$ generated by $Y_{2}$ and $Z_{2}$. With this splitting the diffeomorphism $f_{B}$ verifies that $\mathfrak{g^{cs}}$, $\mathfrak{g}^{cu}$ and $\mathfrak{g}^{c}$ are Lie subalgebras (they are all closed under the Lie bracket) and so $f_{B}$ is in the hypothesis of Theorem \ref{teoA}.
\end{itemize}

\section{Global product structure}
We mentioned at the end of section \ref{spre} that the stable and unstable manifolds exist thanks to the stable manifold theorem, however in the algebraic case, these manifolds have a simpler description. 

Let $M=G/\Gamma$ be a nilmanifold of dimension $d$ and $\mathfrak{g}$ its corresponding Lie algebra. Take $A:M \to M$ an Anosov automorphism such that $\mathfrak{g}$ has a splitting of the form $\mathfrak{g}=\mathfrak{g}^{ss}\oplus \mathfrak{g}^{ws}\oplus \mathfrak{g}^{wu}\oplus \mathfrak{g}^{uu}$. Recall that $\mathfrak{g}^{c}=\mathfrak{g}^{ws}\oplus \mathfrak{g}^{wu}$, $\mathfrak{g}^{cs}=\mathfrak{g}^{ss}\oplus \mathfrak{g}^{c}$ and $\mathfrak{g}^{cu}=\mathfrak{g}^{c}\oplus \mathfrak{g}^{uu}$ are the central, central-stable and central-unstable subspaces respectively, and according to remark \ref{obsliesub} we are assuming they are all Lie subalgebras. Now for every subalgebra $\mathfrak{g}^{\sigma}$ it corresponds a unique connected Lie subgroup $G^{\sigma}$ which is tangent in every point to the corresponding Lie subalgebra (see for example\cite{SW}). Moreover, since the group $G$ is nilpotent we have that the exponential map is a diffeomorphism, and so the subgroups $G^{\sigma}$ are of the form $G^{\sigma}=\exp(\mathfrak{g}^{\sigma})$ (\cite{K}). For $\sigma=s,ss,u,uu,cs,cu,c$, we have the corresponding foliations $\tilde{\mathcal{W}}^{\sigma}_{A}$ which are defined by $\tilde{\mathcal{W}}^{\sigma}_{A}(x)=L_{x}(G^{\sigma})$ for every $x\in G$, where $L_{x}$ is the left multiplication by the element $x$. The sub index make reference to the automorphism $A$ which is where the subgroups comes from, and also the foliations.

To prove that $\tilde{\mathcal{W}}^{\sigma}_{A}$ is truly a foliation is enough to show that if two leaves intersect each other, then they're both the same leaf. 
Let's suppose that $\tilde{\mathcal{W}}^{\sigma}_{A}(x)\cap \tilde{\mathcal{W}}^{\sigma}_{A}(y)\neq \emptyset$. Then there is a point $p$ such that $p \in \tilde{\mathcal{W}}^{\sigma}_{A}(x)=L_{x}(G^{\sigma})$ and $p \in \tilde{\mathcal{W}}^{\sigma}_{A}(y)=L_{y}(G^{\sigma})$. Then there are points $g_{1}$, $g_{2} \in G^{\sigma}$ such that $x.g_{1}=p=y.g_{2}$ and we can write $x=y.g_{2}.g_{1}^{-1}$. 
Given $q \in \tilde{\mathcal{W}}^{\sigma}_{A}(x)$, there is $g_{3}\in G^{\sigma}$ such that $q=x.g_{3}$. 
But then $q=x.g_{3}=(y.g_{2}.g_{1}^{-1}).g_{3}= y.(g_{2}.g_{1}^{-1}.g_{3})\in L_{y}(G^{\sigma})=\tilde{\mathcal{W}}^{\sigma}_{A}(y)$ because $G^{\sigma}$ is a subgroup.
The choice of $q$ was arbitrary so we have proved that $\tilde{\mathcal{W}}^{\sigma}_{A}(x)\subset \tilde{\mathcal{W}}^{\sigma}_{A}(y)$. Analogously we obtain $\tilde{\mathcal{W}}^{\sigma}_{A}(y)\subset \tilde{\mathcal{W}}^{\sigma}_{A}(x)$. 
This proves that $\tilde{\mathcal{W}}^{\sigma}_{A}$ is indeed a foliation. We note by $\mathcal{W}^{\sigma}_{A}$ to the image of $\tilde{\mathcal{W}}^{\sigma}_{A}$ by the canonical projection $G \to G/\Gamma$.

The next proposition is key in our work. Before getting into the proof we have to introduce a definition.

\begin{df}[Global product structure]
We say that two foliations $\mathcal{F}_{1}$ and $\mathcal{F}_{2}$ in $M=G/\Gamma$ have global product structure \textnormal{(GPS)}, if for every pair of points $x,y\in G$ the leaves $\tilde{\mathcal{F}_{1}}(x)$ and $\tilde{\mathcal{F}_{2}}(y)$ intersect in exactly one point.
\end{df}

\begin{prop} \label{ppg} If we are in the same conditions as above, then:
\
\begin{enumerate}

\item The foliations $\mathcal{W}^{cs}_{A}$ and $\mathcal{W}^{uu}_{A}$ have \textnormal{GPS}.
\item The foliations $\mathcal{W}^{cu}_{A}$ and $\mathcal{W}^{ss}_{A}$ have \textnormal{GPS}.
\item The foliations $\mathcal{W}^{c}_{A}$ and $\mathcal{W}^{uu}_{A}$ have \textnormal{GPS} inside the foliation $\mathcal{W}^{cu}_{A}$.
\item The foliations $\mathcal{W}^{c}_{A}$ and $\mathcal{W}^{ss}_{A}$ have \textnormal{GPS} inside the foliation $\mathcal{W}^{cs}_{A}$.
\item The foliations $\mathcal{W}^{s}_{A}$ and $\mathcal{W}^{u}_{A}$ have \textnormal{GPS}.
\end{enumerate}

\begin{proof}
We are going to prove case 1 and see how the rest of the cases follows the same way.
We have to prove that for every pair of points $x,y \in G$, the intersection $\tilde{\mathcal{W}}^{cs}_{A}(x)\cap \tilde{\mathcal{W}}^{uu}_{A}(y)$ is non-empty and is exactly one point. 

\subsection*{Uniqueness of the intersection} 
Suppose that $\tilde{\mathcal{W}}^{cs}_{A}(x)$ and $\tilde{\mathcal{W}}^{uu}_{A}(y)$ intersect in more than one point. First suppose that $y=e$. 
By hypothesis we know there are points $p,q \in \tilde{\mathcal{W}}^{cs}_{A}(x)\cap \tilde{\mathcal{W}}^{uu}_{A}(e)$. 
By definition of the foliations we have that there exist $g^{cs}_{1}, g^{cs}_{2}\in G^{cs}$ and $g^{uu}_{1}, g^{uu}_{2}\in G^{uu}$ such that: 
$$p=xg^{cs}_{1}=g^{uu}_{1} \ \textnormal{and}\ \ q=xg^{cs}_{2}=g^{uu}_{2}$$
Since $G^{uu}$ and $G^{cs}$ are Lie subgroups of $G$, we have 
$$(g^{uu}_{1})^{-1}=(xg^{cs}_{1})^{-1}=(g^{cs}_{1})^{-1}x^{-1} \in G^{uu}$$ Therefore  
$$(g^{uu}_{1})^{-1}g^{uu}_{2}=(g^{cs}_{1})^{-1}x^{-1}xg^{uu}_{2}=(g^{cs}_{1})^{-1}(g^{cs}_{2}) \in G^{cs}\cap G^{uu}=e$$
and we conclude $p=q$.
For the general case let's suppose there exist points $p,q \in \tilde{\mathcal{W}}^{cs}_{A}(x)\cap \tilde{\mathcal{W}}^{uu}_{A}(y)$. Applying $L_{y^{-1}}$ we take the problem to the previous case obtaining  $y^{-1}p, y^{-1}q \in \tilde{\mathcal{W}}^{cs}_{A}(y^{-1}x)\cap \tilde{\mathcal{W}}^{uu}_{A}(e)$. We deduce that $y^{-1}p=y^{-1}q$ if and only if $p=q$.
This proves the uniqueness of the intersection point. 

\subsection*{Existence of the intersection}
To prove the existence, we are going to make a dynamical proof, although it is possible that there is another way to see this in a more algebraic manner. We are going to use the existence of local product structure (LPS) of the foliations in question. This local product structure do exist trivially in a neighbourhood of $e$ and then in every point since the manifold $G$ is homogeneous (for every pair of points $x,y\in G$ there is an isometry $L_{yx^{-1}}$ sending $x$ to $y$). Let $U_{e}$ be a LPS neighbourhood of $e$ and $U_{x}=L_{x}U_{e}$ the LPS neighbourhood of $x\in G$.

Once again we first assume the case $y=e$. We have to prove that $\tilde{\mathcal{W}}^{uu}_{A}(e)$ and $\tilde{\mathcal{W}}^{cs}_{A}(x)$ have non empty intersection. Trivial cases are when $x\in \tilde{\mathcal{W}}^{u}_{A}(e)$ or $x\in \tilde{\mathcal{W}}^{s}_{A}(e)$. Suppose the first one. Then iterating $A$ to the past we get $N>0$ such that $A^{-N}(x)\in U_{e}$. This implies the existence of a point $p \in \tilde{\mathcal{W}}^{uu}_{A}(e)\cap \tilde{\mathcal{W}}^{cs}_{A}(A^{-N}(x))$. Applying $A^{N}$ we get $A^{N}(p) \in \tilde{\mathcal{W}}^{uu}_{A}(e)\cap \tilde{\mathcal{W}}^{cs}_{A}(x)$, the desired intersection. The case $x\in \tilde{\mathcal{W}}^{s}_{A}(e)$ is exactly the same but iterating to the future.

Let's suppose now that $x\in G$ is neither in $G^{s}$ nor in $G^{u}$. Again using the hyperbolicity  we know that exist $p\in \tilde{\mathcal{W}}^{uu}_{A}(e)$ and $n>0$ such that $A^{n}(x)\in U_{p}$. Then, there is $q\in \tilde{\mathcal{W}}^{uu}_{A}(p)\cap \tilde{\mathcal{W}}^{cs}_{A}(A^{n}(x))$. Since $q\in \tilde{\mathcal{W}}^{uu}_{A}(p)$ and $p\in \tilde{\mathcal{W}}^{u}_{A}(e)$, we have that exist $m>0$ such that $A^{-m}(q)\in U_{e}$. Once again there exist $z\in \tilde{\mathcal{W}}^{uu}_{A}(e)\cap \tilde{\mathcal{W}}^{cs}_{A}(A^{-m}(q))$. Applying $A^{m}$ we have that $A^{m}(z)\in \tilde{\mathcal{W}}^{uu}_{A}(e)\cap \tilde{\mathcal{W}}^{cs}_{A}(q)$. But $q\in \tilde{\mathcal{W}}^{cs}_{A}(A^{n}(x))$, so then $\tilde{\mathcal{W}}^{cs}_{A}(q)=\tilde{\mathcal{W}}^{cs}_{A}(A^{n}(x))$. Hence, $A^{m}(z)\in \tilde{\mathcal{W}}^{uu}_{A}(e)\cap \tilde{\mathcal{W}}^{cs}_{A}(A^{n}(x))$. Applying again $A^{-n}$ we obtain $A^{m-n}(z)\in \tilde{\mathcal{W}}^{uu}_{A}(e)\cap \tilde{\mathcal{W}}^{cs}_{A}(x)$ as we wanted to show.

Interchanging the roles of $x$ and $e$ we obtained that $\tilde{\mathcal{W}}^{uu}_{A}(x)$ and $\tilde{\mathcal{W}}^{cs}_{A}(e)$ have non trivial intersection. 

For the general case pick $x,y \in G$. The previous part applying to the point $x^{-1}y$ gives us the existence of a point $p\in \tilde{\mathcal{W}}^{uu}_{A}(x^{-1}y)\cap \tilde{\mathcal{W}}^{cs}_{A}(e)$. Multiplying by $x$ by left, we obtain that $xp \in \tilde{\mathcal{W}}^{uu}_{A}(y)\cap \tilde{\mathcal{W}}^{cs}_{A}(x)$. 

\subsection*{Conclusion of the proof}
Given two points $x,y \in G$, the existence part tell us that $\tilde{\mathcal{W}}^{uu}_{A}(y)\cap \tilde{\mathcal{W}}^{cs}_{A}(x)\neq \emptyset$ and the uniqueness part that the intersection point must be unique, proving the GPS of the foliations $\mathcal{W}^{uu}_{A}$ and $\mathcal{W}^{cs}_{A}$. This proves case 1.

Case 2 is completely analogous to case 1. 

Case 3 reduces to case 2 (and analogously case 4 to case 1): take two points $x,y$ that belong to the same center-stable leaf. We know for case 2 that exist a unique point $p\in \tilde{\mathcal{W}}^{ss}_{A}(x)\cap \tilde{\mathcal{W}}^{cu}_{A}(y)$. Since $\tilde{\mathcal{W}}^{ss}_{A}(x)\subset \tilde{\mathcal{W}}^{cs}_{A}(x)=\tilde{\mathcal{W}}^{cs}_{A}(y)$, we have that $p \in \tilde{\mathcal{W}}^{cs}_{A}(y)\cap \tilde{\mathcal{W}}^{cu}_{A}(y)=\tilde{\mathcal{W}}^{c}_{A}(y)$. This proves that the leaves  $\tilde{\mathcal{W}}^{ss}_{A}(x)$ and $\tilde{\mathcal{W}}^{c}_{A}(y)$ intersect in a unique point $p \in \tilde{\mathcal{W}}^{cs}_{A}(x)=\tilde{\mathcal{W}}^{cs}_{A}(y)$. 

The proof of case 5 is exactly the same as the previous cases, but with trivial center stable.
\end{proof}
\end{prop}

\section{Integrability criterion}

In this section we are going to prove an integrability criterion to obtain dynamically coherence of partially hyperbolic diffeomorphism due to \cite{FPS}. The proofs are similar as in \cite{FPS} but adapted from the torus case to the nilmanifold case. 

\subsection{Canonical projections}

The Proposition \ref{ppg} concerning GPS of the invariant foliations allow us to define canonical projections on the strong stable/unstable foliations. Given $x \in G$, we define the projection $\Pi^{uu}_{x}:G\to \tilde{\mathcal{W}}^{uu}_{A}(x)$ as
$$\Pi^{uu}_{x}(y)=\tilde{\mathcal{W}}^{cs}_{A}(y)\cap \tilde{\mathcal{W}}^{uu}_{A}(x),  \ \ \forall y \in G$$
The map $\Pi^{uu}_{x}$ is the projection on the strong unstable leaf $\tilde{\mathcal{W}}^{uu}_{A}(x)$ through the center-stable leaves $\tilde{\mathcal{W}}_{A}^{cs}$.
Analogously we define the projection $\Pi^{ss}_{x}$.
In the same way using the GPS we can define inside a $\tilde{\mathcal{W}}^{cs}_{A}$-leaf the projection on the strong stable leaves $\tilde{\mathcal{W}}^{ss}_{A}$ through the center leaves $\tilde{\mathcal{W}}^{c}_{A}$.

We note that by the previous definition we get  
$$A \circ \Pi^{uu}_{x}(y)=A(\tilde{\mathcal{W}}^{cs}_{A}(y)\cap \tilde{\mathcal{W}}^{uu}_{A}(x))=\tilde{\mathcal{W}}^{cs}_{A}(A(y))\cap \tilde{\mathcal{W}}^{uu}_{A}(A(x))=\Pi^{uu}_{A(x)}(A(y))$$
Hence $A \circ \Pi^{uu}_{x}=\Pi^{uu}_{A(x)}\circ A$. The same happens with the case $\Pi^{ss}_{x}$. For the sake of simplicity from now on we're going to omit the sub indexes in the projections. 

\subsection{$\sigma$-propperness}

Given $f \in \textnormal{PH}_{A}(G/\Gamma)$, we know from \cite{Fr} that there is a continuous and surjective map $H_{f}:G \to G$ such that $A \circ H_{f}=H_{f}\circ \tilde{f}$, where $\tilde{f}:G \to G$ is a lift of $f$. The stable manifold theorem gives us the existence of two invariant foliations $\tilde{\mathcal{W}}^{ss}_{f}$ and $\tilde{\mathcal{W}}^{uu}_{f}$ tangent to $E^{uu}_{f}$ and $E^{ss}_{f}$ respectively. It's easy to see that $H_{f}(\tilde{\mathcal{W}}^{ss}_{f}(x))\subset \tilde{\mathcal{W}}^{s}_{A}(H_{f}(x))$, then we can project the image of $H_{f}|_{\tilde{\mathcal{W}}^{ss}_{f}(x)}$ to $\tilde{\mathcal{W}}^{ss}_{A}(H_{f}(x))$. We call $H^{\sigma}_{f}:=\Pi^{\sigma} \circ H_{f}$. We are going to note  
$$D^{\sigma}_{f}(x,R)=\{ y\in \tilde{\mathcal{W}}^{\sigma}_{f}(x): d_{\tilde{\mathcal{W}}^{\sigma}_{f}}(x,y)<R  \}
$$ where $ d_{\tilde{\mathcal{W}}^{\sigma}_{f}}(x,y) $ denotes the distance inside the $\tilde{\mathcal{W}}^{\sigma}_{f}$-leaf.

\begin{df}[$\sigma$-proper] For $\sigma=ss,uu$ we say that $f\in \textnormal{PH}_{A}(M)$ is $\sigma$-proper if for every $x\in G$ the map $H^{\sigma}_{f}|_{\tilde{\mathcal{W}}^{\sigma}_{f}(x)}:\tilde{\mathcal{W}}^{\sigma}_{f}(x)\to \tilde{\mathcal{W}}^{\sigma}_{A}(H_{f}(x))$ is uniformly proper. More precisely, for every $R>0$ there exists $\delta>0$ such that $(H^{\sigma}_{f})^{-1}(D^{\sigma}_{A}(H_{f}(x),R))\cap \tilde{\mathcal{W}}_{f}^{\sigma}(x) \subset D^{\sigma}_{f}(x,\delta)$ for every $x\in G$.
\end{df}

\begin{obs} \label{osp1}
In the previous definition, we can take $R=1$ because of uniform hyperbolicity and the compactness of $M$. Then a diffeomorphism $f\in \textnormal{PH}_{A}(M)$ is $\sigma$-proper if and only if exist $\delta>0$ such that$(H^{\sigma}_{f})^{-1}(D^{\sigma}_{A}(H_{f}(x),1))\cap \tilde{\mathcal{W}}^{\sigma}_{f}(x)\subset D^{\sigma}_{f}(x,\delta)$ for every $x\in G$.
\end{obs}

The definition of $\sigma$-properness can be expressed in a different and more geometric way. The next lemma gives the desire equivalence. First we introduce a few definitions.

\begin{df}
Let $f\in \textnormal{PH}_{A}(M)$, we say that $f$ has the condition 
\begin{itemize}
\item[$(I^{\sigma})$] If the function $H^{\sigma}_{f}$ is injective when restricted to $\tilde{\mathcal{W}}^{\sigma}_{f}$ leaves.
\item[$(S^{\sigma})$] If the function $H^{\sigma}_{f}$ is surjective when restricted to the leaves of $\tilde{\mathcal{W}}^{\sigma}_{f}$.
\end{itemize} 
\end{df}
Notice that if we put together this two conditions we get that the map $H^{\sigma}_{f}|_{\tilde{\mathcal{W}}^{\sigma}_{f}(x)}:\tilde{\mathcal{W}}^{\sigma}_{f}(x)\to \tilde{\mathcal{W}}^{\sigma}_{A}(H_{f}(x))$ is a homeomorphism, a much friendly concept.   
\begin{lem} \label{equipro}
If $f\in \textnormal{PH}_{A}(M)$ then, $f$ is $\sigma$-proper if and only if $f$ satisfies properties $(I^{\sigma})$ and $(S^{\sigma})$. 
\end{lem}

\begin{obs}\label{iis}
The proof of the previous lemma can be found in \cite{FPS}. Actually in that paper, the authors divided the lemma in three lemmas and proved that $(I^{\sigma})$ implies $(S^{\sigma})$. Then it's enough to check $(I^{\sigma})$ for being $\sigma$-proper.
\end{obs}

\subsection{Integrability criterion}

Given a subset $K \subset G$ and $R>0$ we call $B_{R}(K)$ the $R$-neighbourhood of $K$, that is, the set of points in $G$ that are less than $R$ from some point in $K$.

\begin{df}[Almost parallel foliations] Given $\mathcal{F}_{1}$, $\mathcal{F}_{2}$ foliations in $G$. We say that the two foliations are \textit{almost parallel} if exists $R>0$ such that for every $x\in G$, there are points $y_{1}, y_{2} \in G:$ such that
\begin{itemize}
\item $\tilde{\mathcal{F}}_{1}(x)\subset B_{R}(\tilde{\mathcal{F}}_{2}(y_{1}))$ and $\tilde{\mathcal{F}}_{2}(y_{1})\subset B_{R}(\tilde{\mathcal{F}}_{1}(x))$
\item  $\tilde{\mathcal{F}}_{2}(x)\subset B_{R}(\tilde{\mathcal{F}}_{1}(y_{2}))$ and $\tilde{\mathcal{F}}_{1}(y_{2})\subset B_{R}(\tilde{\mathcal{F}}_{2}(x))$  
\end{itemize}
\end{df}
It's easy to see that being almost parallel is an equivalence relationship. Moreover, the condition can be expressed in terms of the Hausdorff distance:
$\forall x\in G$, $\exists y_{1}, y_{2}\in G$ such that $D_{H}(\tilde{\mathcal{F}}_{1}(x),\tilde{\mathcal{F}}_{2}(y_{1}))<R$ and $D_{H}(\tilde{\mathcal{F}}_{2}(x),\tilde{\mathcal{F}}_{1}(y_{2}))<R$.

\begin{df}[SADC] We say that a diffeomorphism $f\in \textnormal{PH}_{A}(M)$ is \textbf{strongly almost dynamically coherent} \textnormal{(SADC)} if there exists foliations $\mathcal{F}^{cs}$, $\mathcal{F}^{cu}$ (not necessary invariant) which are respectively transverse to $E^{uu}_{f}$, $E^{ss}_{f}$ and almost parallel to the foliations $\mathcal{W}^{cs}_{A}$, $\mathcal{W}^{cu}_{A}$ respectively.
\end{df}

The previous name (SADC) comes from \cite{Po} where Potrie defines the concept of \textit{almost dynamically coherent} as a partially hyperbolic diffeomorphism with foliations $\mathcal{F}^{cs}$, $\mathcal{F}^{cu}$ transverse to $E^{uu}_{f}$, $E^{ss}_{f}$. In fact in that paper the author proved for dimension 3 that this foliations are almost parallel to $\mathcal{W}^{cs}_{A}$, $\mathcal{W}^{cu}_{A}$. In higher dimension this is not clear, that's why in \cite{FPS} the added the \textit{strong} hypothesis.

\begin{teo}[Integrability criterion \cite{FPS}] \label{tcdi}
\
Assume that $f\in \textnormal{PH}_{A}(M)$ verifies the following conditions:
\begin{itemize}
\item $f$ is $uu$-proper. 
\item $f$ is SADC.
\end{itemize}
Then, the bundle $E^{ss}_{f}\oplus E^{c}_{f}$ is integrable into an $f$-invariant foliation $\mathcal{W}^{cs}_{f}$ that verifies $H_{f}^{-1}(\tilde{\mathcal{W}}^{cs}_{A}(H_{f}(x)))=\tilde{W}^{cs}_{f}(x)$. 
Moreover, $\tilde{\mathcal{W}}^{cs}_{f}$ and $\tilde{\mathcal{W}}^{uu}_{f}$ have global product structure. 
\end{teo}

\begin{proof}
The idea of the proof is pretty clear: take the foliation $\tilde{\mathcal{F}}^{cs}$ given by the SADC property and iterate it backwards by $\tilde{f}$ hoping that in the limit it will converges to the desired foliation. 

We know that $\{H_{f}^{-1}(\tilde{\mathcal{W}}^{cs}_{A}(H_{f}(y))): y\in G\}$ forms an $\tilde{f}$-invariant partition of $G$ that is invariant by deck translations. This is easy to check and is due to the semiconjugacy relation and the fact that $H_{f}$ is $\Gamma$-periodic. Now take the foliation $\tilde{\mathcal{F}}^{cs}$ given by the SADC property. Since it is almost parallel to $\tilde{\mathcal{W}}^{cs}_{A}$ and $H_{f}$ is at bounded Hausdorff distance from the identity we have that $H_{f}(\tilde{\mathcal{F}}^{cs}(x))$ is also at bounded Hausdorff distance from some leaf of $\tilde{\mathcal{W}}^{cs}_{A}$ for every $x\in G$.
\begin{afi} The foliations $\mathcal{F}^{cs}$ and $\mathcal{W}^{uu}_{f}$ have GPS.
\begin{proof}
This is because the properties $(I^{uu})$ and $(S^{uu})$. Given points $x,y \in G$, we consider the subset $Q=G\setminus \tilde{\mathcal{F}}^{cs}(x)$. A Jordan separation like result (see \cite{ABP} Lemma 2.1) says that the $d-cs-1$ homology of $Q$ is non trivial. Here $d=dim(G)$ and $cs=dim(G^{cs})$. 
Since $\tilde{\mathcal{F}}^{cs}(x)$ is at bounded Hausdorff distance of $\tilde{\mathcal{W}}^{cs}_{A}$ we deduce the existence of a non trivial cycle of the $d-cs-1$ homology group $H_{d-cs-1}(Q)$ inside of $\tilde{\mathcal{W}}^{uu}_{A}$. Choosing this cycle sufficently far from $\tilde{\mathcal{F}}^{cs}(x)$, and using properties $(I^{uu})$ and $(S^{uu})$ we deduce the existence of a non trivial cycle contained in $\tilde{\mathcal{W}}^{uu}_{f}(y)$. This gives the desired intersection (see \cite{ABP} for more details).

Due to $(I^{uu})$ property we have that $H_{f}$ is injective when restricted to leaves of $\tilde{\mathcal{W}}^{uu}_{f}$ and also, that for every $y\in G$ we have that $H_{f}(\tilde{\mathcal{W}}^{uu}_{f}(y))$ intersects $\tilde{\mathcal{W}}^{cs}_{A}(e)=G^{cs}$ in only one point. Then, we can define the function $\varphi :\tilde{\mathcal{F}}^{cs}(x)\to \tilde{\mathcal{W}}^{cs}_{A}(e)=G^{cs}$ given by $\varphi (p)=H_{f}(\tilde{\mathcal{W}}^{uu}_{f}(p))\cap G^{cs}$. 

The surjectivity of the function $\varphi$ comes from the intersection proved above. Moreover $\varphi$ is continuous because $H_{f}$ is continuous and the continuous variation of $\tilde{\mathcal{W}}^{uu}_{f}$-leaves. We observed that points in the same $\tilde{\mathcal{F}}^{cs}(x)$-leaf will have the same image by $\varphi$. The continuity of $\tilde{\mathcal{W}}^{uu}_{f}$ and the LPS tell us that $\varphi$ is a covering map. Since $\tilde{\mathcal{F}}^{cs}(x)$ is contractible, the map $\varphi$ is a homeomorphism and we get uniqueness. 
\end{proof}
\end{afi}
The previous claim says that we can see the leaves of $\tilde{\mathcal{F}}^{cs}$ (and then of $\tilde{f}^{-n}(\tilde{\mathcal{F}}^{cs})$) as graphs of functions from $\mathbb{R}^{cs}$ to $\mathbb{R}^{uu}$.  
Since the foliation $\mathcal{F}^{cs}$ is uniformly transversal to $E_{f}^{uu}$ we know there are  LPS boxes of uniform size in $G$: 
There is $\epsilon>0$ s.t. $\forall x \in G$ there is a neighbourhood $V_{x}\supseteq B(x,\epsilon)$ and $C^{1}$-local coordinates $\psi_{x}:\mathbb{D}^{cs}\times \mathbb{D}^{uu} \to V_{x}$ such that:
\begin{itemize}
\item $\psi_{x}(\mathbb{D}^{cs}\times \mathbb{D}^{uu})=V_{x}$
\item For every $y \in B_{\epsilon}(x)\subseteq V_{x}$ we have that if we call $W^{x}_{n}(y)$ to the connected component of $V_{x}\cap \tilde{f}^{-n}(\tilde{\mathcal{F}}^{cs}(\tilde{f}^{n}(y)))$ that contains $y$ then
$$\psi_{x}^{-1}(W^{x}_{n}(y))=\textnormal{graph}(h^{x,y}_{n})$$ 
where $h^{x,y}_{n}:\mathbb{D}^{cs}\to \mathbb{D}^{uu}$ is a $C^{1}$ function with bounded first derivatives.  
\end{itemize}
This way we get that the set $\{h^{x,y}_{n}\}_{n\in \mathbb{N}}$ is precompact in the space of Lipschitz functions $\mathbb{D}^{cs}\to \mathbb{D}^{uu}$ (\cite{HPS}). Therefore the leaves of $\tilde{f}^{-n}(\tilde{\mathcal{F}}^{cs})$ have convergent sub-sequences. From this point we have to deal with two problems: the first one is that a priori there could be a leaf with more than one limit, and second, that in the limit, different leaves might merge. We will handle the two problems in the same way.

For every $y \in B_{\epsilon}(x)$, we call $\mathcal{J}^{x}_{y}$ to the set of indices such that for every $\alpha \in \mathcal{J}^{x}_{y}$ there is a Lipschitz function $h^{x,y}_{\infty,\alpha}:\mathbb{D}^{cs}\to \mathbb{D}^{uu}$ and a subsequence $n_{j}\to +\infty$ such that 
$$h^{x,y}_{\infty,\alpha}=\lim_{j\to +\infty}h^{x,y}_{n_{j}}$$

Every $h^{x,y}_{\infty,\alpha}$ has its corresponding graph, and we note $W^{x}_{\infty,\alpha}(y)$ to the image by $\psi_{x}$ of this graph.

\begin{afi}
For every $z\in B(x,\epsilon)$ and every $\alpha \in \mathcal{J}^{x}_{z}$, we have that $H_{f}(W^{x}_{\infty,\alpha}(z))\subseteq \tilde{\mathcal{W}}^{cs}_{A}(H_{f}(z))=L_{H_{f}(z)}(G^{cs})$. 
\begin{proof}
Take $z\in B(x,\epsilon)$ and $\alpha \in \mathcal{J}^{x}_{z}$. Then by hypothesis there is subsequence $n_{j}\to +\infty$ such that $W^{x}_{n_{j}}(z)\to W^{x}_{\infty,\alpha}(z)$. Given $y\in W^{x}_{\infty,\alpha}(z)$, we want to prove that $H_{f}(y)\in \tilde{\mathcal{W}}^{cs}_{A}(H_{f}(z))$. Call $z_{n_{j}}=\tilde{\mathcal{F}}^{uu}(y)\cap W^{x}_{n_{j}}(z)$. Then $\tilde{f}^{n_{j}}(z_{n_{j}})\in \tilde{\mathcal{F}}^{cs}(\tilde{f}^{n_{j}}(z))$ and $z_{n_{j}}\to y$ when $j \to \infty$. If $H_{f}(y)=H_{f}(z)$ we're done. Suppose by the contrary that $H_{f}(z)\neq H_{f}(y)$. Then $H_{f}(z_{n_{j}})\to H_{f}(y)\neq H_{f}(z)$ by continuity of $H_{f}$. Note that $z$ and $z_{_{j}}$ belong to the same leaf $\tilde{\mathcal{F}}^{cs}$, and the same for $\tilde{f}^{n_{j}}(z)$ and $\tilde{f}^{n_{j}}(z_{n_{j}})$. Since $\mathcal{F}$ is almost parallel to $\mathcal{W}^{cs}_{A}$ we have that $\tilde{f}^{n_{j}}(z)$ and $\tilde{f}^{n_{j}}(z_{n_{j}})$ are at uniformly bounded distance respect to the $G^{uu}$ direction and the same happens with $H_{f}(\tilde{f}^{n_{j}}(z))$ and $H_{f}(\tilde{f}^{n_{j}}(z_{n_{j}}))$ because $H_{f}$ is $C^{0}$-close to the identity. By semiconjugacy we get that $A^{n_{j}}(H_{f}(z_{n_{j}}))$ and $A^{n_{j}}(H_{f}(z))$ are at uniformly bounded distance respect to the $G^{uu}$ direction. Since $A$ is Anosov this can happen only if $H_{f}(y)\in \tilde{\mathcal{W}}^{cs}_{A}(H_{f}(z))$, as we wanted to see.  
\end{proof}
\end{afi}

We are going to solve the two problems mentioned above in the same way. Suppose first that $z \in B(x,\epsilon)$ has two different limits $W^{x}_{\infty,\alpha}(z)$ and $W^{x}_{\infty,\beta}(z)$. Then there are points $z_{1}\in W^{x}_{\infty,\alpha}(z)$ and $z_{2}\in W^{x}_{\infty,\beta}(z)$ that belong to the same $\mathcal{W}^{uu}_{f}$-leaf. The previous claim implies that $H_{f}(z_{1})$ and $H_{f}(z_{2})$ belong to $\mathcal{W}^{cs}_{A}(H_{f}(z))$ and this can happen if and only if $H^{uu}_{f}(z_1)=H^{uu}_{f}(z_{2})$ which contradicts the injectivity of $H^{uu}_{f}$.

For the second problem we manage the same way. Lets suppose there are points $z_1 \neq z_2$ in $B(x,\epsilon)$ such that their limits $W^{x}_{\infty,\alpha}(z_1)$ and $W^{x}_{\infty,\beta}(z_2)$ have non empty intersection. Then we get two points $y_1 \in W^{x}_{\infty,\alpha}(z_1)$ and $y_2 \in W^{x}_{\infty,\beta}(z_{2})$ inside the same $\mathcal{W}^{uu}_{f}$ leaf. Again the previous claim said that $H^{uu}_{f}(z_1)=H^{uu}_{f}(z_{2})$ and this contradicts the injectivity of $H^{uu}_{f}$.  

To sum up, we obtained that for every $x\in G$ and every $y\in B(x,\epsilon)$, the limit $W^{x}_{\infty}(y)$ of the $W^{x}_{n}(y)$ leaves is unique, and for every pair of points $y,z \in B(x,\epsilon)$, their limits are disjoint or coincide. These limits are also $f$-invariant. To get that is truly a foliation, it's enough to observe the following: Given two points $z,w\in B(x,\epsilon)$, we have that $W^{x}_{\infty}(z)$ and $\mathcal{W}^{uu}_{f}(w)$ intersect in a unique point. Since the leaves of $\mathcal{W}^{uu}_{f}$ varies continuously and the plaques of $W^{x}_{\infty}$ either coincide or are disjoint, we get a continuous function from $\mathbb{D}^{cs}\times \mathbb{D}^{uu}$ to a neighbourhood of $x$ which sends horizontal disks to $W^{x}_{\infty}$-plaques. This proves that the plaques form a foliation. Since the leaves of the foliations are tangent to small cones around the $E^{cs}_{f}$ direction and also are $f$-invariant, we get that the foliation is tangent to $E^{cs}_{f}$.

Finally, we observe that the foliation $\tilde{\mathcal{W}}^{cs}_{f}$ has the same properties that $\tilde{\mathcal{F}}^{cs}$ thus we have GPS between $\mathcal{W}^{cs}_{f}$ and $\mathcal{W}^{uu}_{f}$.
\end{proof}

A symmetric statement holds for $f$ being $ss$-proper, so we obtain the following corollary.

\begin{cor} \label{ccdi}
If $f \in \textnormal{PH}_{A}(G/\Gamma)$ verifies the following conditions:
\begin{itemize}
\item $f$ is SADC.
\item $f$ is $uu$ and $ss$-proper.
\end{itemize}
Then $f$ is dynamically coherent and center fibered. 
\end{cor}

\begin{obs}
We want to remark that in every integrability theorem, the uniqueness of the leaves is always a local problem. In this case, the solution is global since it uses the asymptotic behaviour of the leaves of the SADC foliation for $f$.
\end{obs}

\section{proof of the theorem}
To obtain the main theorem, we have to prove that SADC and $\sigma$-properness ($\sigma=ss,uu$) are $C^{1}$ open and closed properties among partially hyperbolic diffeomorphisms in $G/\Gamma$ isotopic to $A$. Then we can apply corolary \ref{ccdi} to a whole connected component as long as it contains diffeomorphism with such properties. The proofs are as \cite{FPS} adapted to nilmanifold case.

\begin{prop} \label{tsadcayc} Being SADC is a $C^{1}$ open and closed property among $\textnormal{PH}_{A}(G/ \Gamma)$.
\end{prop} 
\begin{proof}
Open is trivial since the same foliation works because the continuity of the $E^{ss}$ and $E^{uu}$ bundles: take the foliations $\mathcal{F}^{cs}_{f}$, $\mathcal{F}^{cu}_{f}$ given by the SADC property. These are transverse to $E^{uu}_{f}$, $E^{ss}_{f}$ and almost parallel to $\mathcal{W}^{cs}_{A}$, $\mathcal{W}^{cu}_{A}$ respectively. Then $\sphericalangle(\mathcal{F}^{cs}_{f},E^{uu}_{f})>\epsilon$ for every $x\in G$.
Then there is $\mathcal{U}(f)$ neighbourhood of $f$ in the $C^{1}$ topology such that for every $g \in \mathcal{U}(f)$ we have $\sphericalangle (E^{uu}_{g}(x),E^{uu}_{f}(x))<\frac{\epsilon}{2}$ for every $x\in G$. Take $\mathcal{F}^{cs}_{g}=\mathcal{F}^{cs}_{f}$, then
\begin{eqnarray*}
(\mathcal{F}^{cs}_{g},E^{uu}_{g})+\frac{\epsilon}{2} & > & \sphericalangle (\mathcal{F}^{cs}_{g}(x),E^{uu}_{g}(x))+\sphericalangle (E^{uu}_{g}(x),E^{uu}_{f}(x))\\
&\geq & \sphericalangle (\mathcal{F}^{cs}_{g}(x),E^{uu}_{f}(x))\\
&=&\sphericalangle (\mathcal{F}^{cs}_{f}(x),E^{uu}_{f}(x))>\epsilon >0
\end{eqnarray*}
This implies that $\sphericalangle (\mathcal{F}^{cs}_{g}(x),E^{uu}_{g}(x))>\frac{\epsilon}{2}$ for every $x \in G$.
Thus every $g\in \mathcal{U}(f)$ has foliations $\mathcal{F}^{cs}_{g}$, $\mathcal{F}^{cu}_{g}$ transverse to $E^{uu}_{g}$, $E^{ss}_{g}$. Hence each $g\in \mathcal{U}(f)$ verifies SADC.

For closeness, first note that since $f_{n}$ is isotopic to $A$, it fixes the class of foliations almost parallel to any $A$-invariant subgroup. On the other hand, take $f_{n}\rightarrow f$ in the $C^{1}$ topology such that every $f_{n}$ is SADC. Call $E^{cs}_{n}=E^{ss}_{f_{n}}\oplus E^{c}_{f_{n}}$. By the $C^{1}$ convergence we have $E^{cs}_{n}\rightarrow E^{cs}_{f}$, $E^{uu}_{n}\rightarrow E^{uu}_{f}$. Let $\alpha=\sphericalangle (E^{cs}_{f},E^{uu}_{f})$ (minimum bound of the angle). Now since $E^{cs}_{n}\rightarrow E^{cs}_{f}$ there is $N>0$ such that $\sphericalangle (E^{uu}_{f},E^{cs}_{N})>\frac{\alpha}{2}$. Take $\mathcal{F}^{cs}_{N}$ foliation uniformly transverse to $E^{uu}_{N}$. Then there is $m>0$ such that $f_{N}^{-m}(\mathcal{F}^{cs}_{N})$ is contained in a cone centred at $E^{cs}_{N}$ of radius $\frac{\alpha}{2}$. Thus $f_{N}^{-m}(\mathcal{F}^{cs}_{N})$ is uniformly transverse to $E^{uu}_{f}$. This finish the proof.
\end{proof}

Before looking at the next proposition let us make the following classical remark. 

\begin{obs} \label{rkhyp}
For $f\in \textnormal{PH}(M)$, there exist constants $1<\lambda_{f}<\Delta_{f}$ such that in a $C^{1}$-neighbourhood $\mathcal{U}$ of $f$ we have:
$$D^{uu}_{g}(\tilde{g}(x),\lambda_{f}R)\subset \tilde{g}(D^{uu}_{g}(x,R))\subset D^{uu}_{g}(\tilde{g}(x),\Delta_{f}R)
$$ for every $g\in \mathcal{U}$, $x\in \tilde{M}$ and $R>0$. Analogously for $D^{ss}$ by applying $\tilde{g}^{-1}$.
\end{obs}

\begin{prop} \label{pspopen}
Being $\sigma$-proper is a $C^{1}$-open property in $\textnormal{PH}_{A}(G/\Gamma)$. 

\begin{proof}
Given $f\in \textnormal{PH}_{A}(G/\Gamma)$, we must find a neighbourhood $\mathcal{U}(f)$ in the $C^{1}$ topology such that every $g\in \mathcal{U}(f)$ is $\sigma$-proper. Remark \ref{osp1} said that it's enough to find a neighbourhood $\mathcal{U}(f)$ and $R_{1}>0$ such that for every  $g\in \mathcal{U}(f)$ and $x\in G$:
$$ (H^{\sigma}_{g})^{-1}(D_{A}^{\sigma}(H_{g}(x),1))\cap \tilde{\mathcal{W}}_{g}^{\sigma}(x)\subseteq D_{g}^{\sigma}(x,R_{1})
$$
Since $f$ is $\sigma$-proper, we know that $H^{\sigma}_{f}|_{\tilde{\mathcal{W}}^{\sigma}_{f}(x)}:\tilde{\mathcal{W}}^{\sigma}_{f}(x)\to \tilde{\mathcal{W}}^{\sigma}_{A}(H_{f}(x)) $ is a homeomorphism. Then, there is $R_{1}>0$ such that $$H^{\sigma}_{f}(D^{\sigma}_{f}(x,R_{1})^{c})\cap D^{\sigma}_{A}(H_{f}(x),2)=\emptyset $$
Call $A^{\sigma}_{r,R,g}(x)$ the annulus $D^{\sigma}_{g}(x,R)\setminus D^{\sigma}_{g}(x,r)$ for $R>r>0$. Then for $R_{2}>\Delta_{f}R_{1}$ we have that
$$H^{\sigma}_{f}(A^{\sigma}_{R_{1},R_{2},f}(x))\cap D^{\sigma}_{A}(H_{f}(x),2)=\emptyset
$$ where we take $\Delta_{f}>1$ like in the previous remark. We observe now that $h_{f}$ is uniformly continuous and $H_{f}$ is a lift of $h_{f}$, then $H_{f}$ is uniformly continuous too. Then there is $\epsilon_{1}>0$ such that if $d(x,y)<\epsilon_{1}$ then $d(H_{f}(x),H_{f}(y))<1/4$. Take the following $C^{1}$-neighbourhoods:
\begin{itemize}
\item From uniform hyperbolicity we have there is $\mathcal{U}_{1}(f)$ such that the constants $\Delta_{f}$ and $\lambda_{f}$ are uniform in $\mathcal{U}_{1}(f)$ (see Remark \ref{rkhyp}).
\item The continuous variation of the leaves in the $C^{1}$ topology said that for every $\epsilon_{1}>0$ and $R_{2}>0$, there is $\mathcal{U}_{2}(f)$ and $\delta>0$ such that for every $g\in \mathcal{U}_{2}(f)$ and every pair of points $x,y$ with $d(x,y)<\delta$ we have $d_{C^{1}}(D^{\sigma}_{g}(x,R_{2}),D^{\sigma}_{g}(y,R_{2}))<\epsilon_{1}$.
\item Take $\mathcal{U}_{3}(f)=\{g\in \textnormal{PH}_{A}(G/\Gamma): d_{C^{0}}(H_{f},H_g)<1/4\}$.
\end{itemize}
Finally take $\mathcal{U}_{f}=\mathcal{U}_{1}(f)\cap \mathcal{U}_{2}(f)\cap \mathcal{U}_{3}(f)$.
Now, let $g\in \mathcal{U}(f)$ and $x,y$ such that $y\in A^{\sigma}_{R_{1},R_{2},g}(x)$. Then there is $z\in A^{\sigma}_{R_{1},R_{2},f}(x)$ such that $d(z,y)<\epsilon_{1}$ and from uniform continuity we get $d(H_{f}(z),H_{f}(y))<1/4$. Since $z\in A^{\sigma}_{R_{1},R_{2},f}(x)$ and $d(H_{f}(z),H_{f}(y))<1/4$, applying the triangular inequality we obtain:
\begin{eqnarray*}
2&<&d(H_{f}^{\sigma}(z),H^{\sigma}_{f}(x))\leq d(H^{\sigma}_{f}(z),H^{\sigma}_{f}(y))+d(H^{\sigma}_{f}(y),H^{\sigma}_{f}(x)) \\ 
&\leq & 1/4+ d(H^{\sigma}_{f}(y),H^{\sigma}_{f}(x))
\end{eqnarray*}
Therefore $d(H^{\sigma}_{f}(y),H^{\sigma}_{f}(x))>2-1/4$. Once again the triangular inequality gives:
\begin{eqnarray*}
2-1/4 &<& d(H^{\sigma}_{f}(y),H^{\sigma}_{f}(x)) \\
&\leq& d(H^{\sigma}_{f}(y),H^{\sigma}_{g}(y))+d(H^{\sigma}_{g}(y),H^{\sigma}_{g}(x))+d(H^{\sigma}_{g}(x),H^{\sigma}_{f}(x)) \\
&\leq & 1/4 +d(H^{\sigma}_{g}(y),H^{\sigma}_{g}(x))+1/4
\end{eqnarray*} and we conclude that $d(H^{\sigma}_{g}(y),H^{\sigma}_{g}(x))>2-3/4>1$, which means 
\begin{equation} \label{ecanillo}
H^{\sigma}_{g}(A^{\sigma}_{R_{1},R_{2},g}(x))\cap D^{\sigma}_{A}(H_{g}(x),1)=\emptyset
\end{equation} 
Finally this implies
$$(H^{\sigma}_{g})^{-1}(D_{A}^{\sigma}(H_{g}(x),1))\cap \tilde{\mathcal{W}}_{g}^{\sigma}(x)\subseteq D_{g}^{\sigma}(x,R_{1})
$$
If it weren't the case, there will be $y\in \tilde{\mathcal{W}}^{\sigma}_{g}(x)$ such that $H_{g}^{\sigma}(y)\in D^{\sigma}_{A}(H_{g}(x),1)$ but $y\notin D^{\sigma}_{g}(x,R_{2})$. By the choice of $\Delta_{f}$ we know that there is $n\in \mathbb{Z}$ such that $\tilde{g}^{n}(y)\in A^{\sigma}_{R_{1},R_{2},g}(\tilde{g}^{n}(x))$ and $H^{\sigma}_{g}(\tilde{g}^{n}(x))\in D^{\sigma}_{g}(\tilde{g}^{n}(x),1)$. This contradicts \ref{ecanillo} above. 
\end{proof}

\end{prop}

The previous proposition shows that $\sigma$-properness is $C^{1}$-open in $\textnormal{PH}_{A}(M)$. To finish the theorem we have to prove that it is also a $C^{1}$-closed property. This is the most difficult part of the theorem. For the proof we are going to use Theorem \ref{tcdi} so we have to add the hypothesis of SADC. This doesn't represent any problem since we already know that SADC is open and closed by Proposition \ref{tsadcayc}. 

\begin{teo} \label{tsadcypayc}
Being SADC and $\sigma$-proper is a $C^{1}$-closed property in $\textnormal{PH}_{A}(M)$.
\end{teo}

\begin{proof}
Take $\{f_{k}\}\subset \textnormal{PH}_{A}(M)$ with $f_{k}\to f$ in the $C^{1}$ topology, such that $f_{k}$ is SADC and $\sigma$-proper for every $k \in \mathbb{N}$. By Proposition \ref{tsadcayc} we know that $f$ is also SADC. We have to prove that $f$ is $\sigma$-proper. We can assume that $\sigma=uu$, case $\sigma=ss$ is symmetric.

Note that every $f_{k}$ is in the hypothesis of Theorem \ref{tcdi}, then for every $k\in \mathbb{N}$ there is a $f_{k}$-invariant foliation $\mathcal{W}^{cs}_{f_k}$ tangent to $E^{ss}_{f_k}\oplus E^{c}_{f_k}$ which verifies
\begin{equation}
\tilde{\mathcal{W}}^{cs}_{f_k}(x)=H_{f_k}^{-1}(\tilde{\mathcal{W}}^{cs}_{A}(H_{f_k}(x)))
\end{equation} 
From now on, we are going to note the sub indexes by $k$ instead of $f_{k}$, i.e. $H^{uu}_{k}=\Pi^{uu}\circ H_{f_{k}}$. With this notation the previous property is equivalent to 
\begin{equation} \label{ecscf}
H^{uu}_{k}(x)=H^{uu}_{k}(y) \ \ \textnormal{if and only if} \ \ y\in \tilde{\mathcal{W}}^{cs}_{k}(x).
\end{equation}
\begin{afi}
Given $\epsilon >0$, there exists $\delta>0$, a cone field $\mathscr{C}^{uu}$ around $E^{uu}_{f}$ and  $k_0$ such that if $k \geq k_0$ and $D$ is a disk tangent to $\mathscr{C}^{uu}$ of internal radius larger than $\epsilon$ and centered at $x$, then
$$D^{uu}_{A}(H_k(x),\delta)\subset H^{uu}_{k}(D)
$$ 
\end{afi}
\begin{proof}
This is because $f_k \to f$ in the $C^{1}$ topology implies that $E^{\sigma}_{k}\to E^{\sigma}_{f}$ for every $\sigma$. Then $f$ has a finite cover of LPS boxes $B$ of size smaller than $\epsilon$ such that for $k \geq k_{0}$ large enough, these are LPS boxes for $f_{k}$ too. We can take these boxes $B$ small enough to have:
\begin{itemize}
\item The boxes $2B$ and $3B$ are also LPS boxes for $f_k$.
\item For every $B$ of the covering and every disk $D$ tangent to $\mathscr{C}^{uu}$ of internal radius larger than $\epsilon$ and centered at a point $x\in B$ we have that $D$ intersects in a unique point in $3B$ every center-stable plaque of $\mathcal{W}^{cs}_{k}$ which intersects $2B$.
\end{itemize} 
The previous condition plus equation \ref{ecscf} implies that $H^{uu}_{k}(2B)\subset H^{uu}_{k}(D)$.
Using the injectivity of $H_{k}$ when restricted to $\tilde{\mathcal{W}}^{uu}_{k}$ leaves, we have that given a connected component $2B$ of the lift of a LPS box we have $int(H^{uu}_{k}(2B)) \neq \emptyset
$ and every point $x \in B$ lies in the interior of $H^{uu}_{k}(2B)$. Since there are finite boxes, there is a uniform $\delta$ such that  $H^{uu}_{k}(B)$ is at bounded $\delta$ distance from the boundary of  $H^{uu}_{k}(2B)$ independently of the box $B$.
We deduce that every disk $D$ of internal radius $\epsilon$ and centered at $x$ and tangent to a small cone around $E^{uu}_{f}$ verifies that $H^{uu}_{k}(D)$ contains $D^{uu}_{A}(H_{k}(x),\delta)$ as desired.  
\end{proof}

\begin{afi}
For $k$ sufficiently large enough and for every pair of points $x,y \in G$, we have that $\tilde{\mathcal{W}}^{uu}_{f}$ and $\tilde{\mathcal{W}}^{cs}_{k}$ have non trivial intersection. 
\end{afi}

\begin{proof}
Given two points $x,y \in G$, take a curve in $\tilde{\mathcal{W}}^{uu}_{A}(H_{k}(x))$ that connects 
$H^{uu}_{k}(x)$ and $H^{uu}_{k}(y)$. We can assume that the curve is contained in $m$ local product structure boxes $B_{1},...,B_{m}$ like in the previous claim such that $B_{i}\cap B_{i+1}\neq \emptyset$. Taking $D=D^{uu}_{f}(x,\epsilon)$ we have that $H^{uu}_{k}(D^{uu}_{f}(x,\epsilon))$ contains $B_{1}$ and then $H^{uu}_{k}(D^{uu}_{f}(x,2\epsilon))$ contains $B_{1}\cup B_{2}$. Inductively we have that $H^{uu}_{k}(D^{uu}_{f}(x,m\epsilon))$ contains $B_{1}\cup...\cup B_{m}$ and in particular $H^{uu}_{k}(y)$. This means there is $p \in D^{uu}_{f}(x,m\epsilon)$ such that $H^{uu}_{k}(p)=H^{uu}_{k}(y)$ or equivalently $H_{k}(p) \in \tilde{\mathcal{W}}^{cs}_{A}(H_{k}(y))$. By equation \ref{ecscf} we have that $p \in D^{uu}_{f}(x,m\epsilon)\cap \tilde{\mathcal{W}}^{cs}_{k}(y)$. 
\end{proof}

\begin{afi}
For $k$ sufficiently large enough, the foliations $\tilde{\mathcal{W}}^{uu}_{f}$ and $\tilde{\mathcal{W}}^{cs}_{k}$ have GPS. Equivalently, the map $H^{uu}_{k}|_{\tilde{\mathcal{W}}^{cs}_{f}(x)}:\tilde{\mathcal{W}}^{cs}_{f}(x) \to \tilde{\mathcal{W}}^{uu}_{A}(H_{f}(x))$ is a homeomorphism.
\end{afi}
\begin{proof}
By the previous claim, we only have to prove that the intersection between $\tilde{\mathcal{W}}^{uu}_{f}(x)$ and $\tilde{\mathcal{W}}^{cs}_{k}(y)$ is unique for every pair of points $x,y$. Since the leaf $\tilde{\mathcal{W}}^{uu}_{f}(x)$ intersects transversely  $\tilde{\mathcal{W}}^{cs}_{k}(y)$ for every $x,y$ and $H_{k}(\tilde{\mathcal{W}}^{cs}_{k}(y))=\tilde{\mathcal{W}}^{cs}_{A}(H_{k}(y))$ we have that $H_{k}(\tilde{\mathcal{W}}^{uu}_{f}(x))$ is topologically transverse to $\tilde{\mathcal{W}}^{cs}_{A}(H_{k}(y))$. This implies that  
$$\Pi^{uu}:H_{k}(\tilde{\mathcal{W}}^{uu}_{f}(x))\to \tilde{\mathcal{W}}^{uu}_{A}(H_{f}(x))
$$
is a covering and since $H_{k}(\tilde{\mathcal{W}}^{uu}_{f}(H_{f}))$ is contractible we know it's injective. This proves that $H_{k}^{uu}$ when restricted to $\tilde{\mathcal{W}}^{uu}_{f}(x)$ is a homeomorphism onto $\tilde{\mathcal{W}}^{uu}_{A}(H_{f}(x))$ and equivalently the foliations $\tilde{\mathcal{W}}^{uu}_{f}$ and $\tilde{\mathcal{W}}^{cs}_{k}$ have GPS.
\end{proof} 
To finish the proof of the theorem we must prove there is $R>0$ such that
$$ (H^{uu}_{f})^{-1}(D^{uu}_{A}(H_{f}(x),1))\cap \tilde{\mathcal{W}}^{uu}_{f}(x)\subset D^{uu}_{f}(x,R) \ \ \ \forall x\in G
$$
Fix $x\in G$. We know that $d_{C^{0}}(H_{k},H_{f})<K_{0}$. The previous claim and the fact that $f_{k}$ is CF implies that the restriction of $H^{uu}_{k}$ to $\tilde{\mathcal{W}}^{uu}_{f}(x)$ is a homeomorphism onto $\tilde{\mathcal{W}}^{uu}_{A}(H_{f}(x))$, then there is $R_{1}>0$ such that 
$$H^{uu}_{k}((D^{uu}_{f}(x,R_{1}))^{c})\cap D^{uu}_{A}(H_{k}(x),1+2K_{0})=\emptyset
$$
Take $y \in D^{uu}_{f}(x,R_{1})^{c}$ then applying the triangular inequality we obtain  
\begin{eqnarray*}
1+2K_{0} & < & d(H^{uu}_{k}(x),H^{uu}_{k}(y)) \\
& \leq & d(H^{uu}_{k}(x),H^{uu}_{f}(x))+d(H^{uu}_{f}(x),H^{uu}_{f}(y))+d(H^{uu}_{f}(y),H^{uu}_{k}(y)) \\
& < &  K_{0} +d(H^{uu}_{f}(x),H^{uu}_{f}(y))+ K_{0}
\end{eqnarray*}
Thus $d(H^{uu}_{f}(x),H^{uu}_{f}(y)) >1$ and therefore we get 
$$H^{uu}_{k}((D^{uu}_{f}(x,R_{1}))^{c})\cap D^{uu}_{A}(H_{k}(x),1)=\emptyset
$$
which implies that
$$(H^{uu}_{f})^{-1}(D^{uu}_{A}(H_{f}(x),1))\cap \tilde{\mathcal{W}}^{uu}_{f}(x)\subset D^{uu}_{f}(x,R_{1})
$$
We have proved that the function $\varphi$ is well defined where $$\varphi(x)=\inf \{ R>0: (H^{uu}_{f})^{-1}(D^{uu}_{A}(H_{f}(x),1))\cap \tilde{\mathcal{W}}^{uu}_{f}(x)\subset D^{uu}_{f}(x,R)\}$$ If we prove that $\varphi$ is uniformly bounded in $G$ then by remark \ref{osp1} we get the theorem. Since $\varphi$ is $\Gamma$-periodic, it's enough to restrict ourselves to points in a fundamental domain, which is compact. Thus, it is enough to show that if $x_{n}\to x$ then $\varphi(x_{n})\leq \varphi (x)$. 
To prove this, note that $H^{\sigma}_{f}(D^{\sigma}_{f}(x,\varphi(x)))$ contains $D^{\sigma}_{A}(H_{f}(x),1)$. Now for every $\epsilon >0$ we can find $\delta>0$ such that
$$D^{\sigma}_{A}(H_{f}(x),1+\delta)\subset \Pi^{\sigma} \circ H_{f}(D^{\sigma}_{f}(x),\varphi(x)+\epsilon)
$$
By continuous variation of the $\tilde{\mathcal{W}}^{\sigma}$-leaves and the continuity of the functions $\Pi^{\sigma}$ and $H_{f}$ we deduce that for $n$ large enough $H^{\sigma}_{f}(D^{\sigma}_{f}(x_n,\varphi(x)+\epsilon)$ contains $D^{\sigma}_{A}(H_{f}(x_n),1)$. This shows that $\limsup \varphi(x_n)\leq \varphi(x)+\epsilon$. Since the choice of $\epsilon >0$ was arbitrary, we get the desire result. 
\end{proof}

From the previous results we obtain the following theorem.

\begin{teo}
Let $f\in \textnormal{PH}_{A}(M)$ be a diffeomorphism in the same connected component of a partially hyperbolic diffeomorphism $g$ which is $\sigma$-proper ($\sigma=uu,ss$) and SADC. 
Then $f$ is dynamically coherent and center fibered. 
\end{teo}

\begin{proof}
Propositions \ref{tsadcayc}, \ref{pspopen} and Theorem \ref{tsadcypayc} tell us that $\sigma$-proper and SADC are open and closed properties in the $C^{1}$ topology among $\textnormal{PH}_{A}(M)$. This implies that any diffeomorphism $f$ in the same connected component of a partially hyperbolic diffeomorphism $g$ which is $\sigma$-proper for $\sigma=ss,uu$ and SADC, is in hypothesis of Theorem \ref{tcdi}. Then $f$ is dynamically coherent and center fibered.
\end{proof}

\begin{proof}[Proof of Theorem \ref{teoA}]
To prove the theorem it's enough to show that any diffeomorphism $f \in \textnormal{PH}_{A}^{0}(M)$ which is dynamically coherent and center fibered, it has to be SADC and $\sigma$-proper for $\sigma=ss,uu$. 

Take a diffeomorphism $f$ in this setting. Since $f$ is dynamically coherent we have the existence of center-stable and center-unstable foliations $\mathcal{W}^{cs}_{f}$, $\mathcal{W}^{cu}_{f}$.
We know that this foliations are uniformly transverse to $E^{uu}_{f}$ and $E^{ss}_{f}$ respectively. To prove that $f$ is SADC it remains to show that $\mathcal{W}^{cs}_{f}$ and $\mathcal{W}^{cu}_{f}$ are almost parallel to the center-stable and center-unstable foliations of $A$. We are going to prove it for the $cs$ case, the other one is symmetric.

First of all note that we can think in the center-stable leaf $\tilde{\mathcal{W}}^{cs}_{f}(x)$ as the set of points $y \in G$, which can be connected to $x$, concatenating paths each one contained in either $\tilde{\mathcal{W}}^{ss}_{f}$ or $\tilde{\mathcal{W}}^{c}_{f}$. That means, for every $y\in G$ there are finite paths $\alpha_{1}, \dots,\alpha_{n}$, such that $\alpha_{i}\subset \tilde{\mathcal{W}}^{ss}_{f}$ or $\alpha_{i}\subset \tilde{\mathcal{W}}^{c}_{f}$, and also the path $\alpha_{1}*\dots * \alpha_{n}$ connects $x$ to $y$. Suppose that $y \in \tilde{\mathcal{W}}^{cs}_{f}(x)$. We can assume by simplicity that $n=2$ (and then applying induction argument). Then, there is a point $z\in \tilde{\mathcal{W}}^{ss}_{f}(x)\cap \tilde{\mathcal{W}}^{c}_{f}(y)$, and two paths $\alpha_{1}\subset \tilde{\mathcal{W}}^{ss}_{f}(x)$ and $\alpha_{2}\subset \tilde{\mathcal{W}}^{c}_{f}(y)$ such that the path $\alpha_{1}*\alpha_{2}$ connects $x$ and $y$ through the point $z$. Since $z\in \tilde{\mathcal{W}}^{ss}_{f}(x)\cap \tilde{\mathcal{W}}^{c}_{f}(y)$ and $f$ is CF we have that $H(y)\subset \tilde{\mathcal{W}}^{c}_{A}(H(z))=L_{H(z)}G^{c}$. Then there exists $g^{c}\in G^{c}$ such that $H(y)=H(z)g^{c}$. On the other hand $H(z)\in \tilde{\mathcal{W}}^{s}_{A}(H(x))=L_{H(x)}G^{s}$, hence there is $g^{s}\in G^{s}$ such that $H(z)=H(x)g^{s}$. We conclude that $H(y)=H(x)g^{s}g^{c}\in L_{H(x)}G^{cs}=\tilde{\mathcal{W}}^{cs}_{A}(H(x))$ because $G^{cs}$ is a subgroup. Since the point $y\in G$ was arbitrary we obtain $H_{f}(\tilde{\mathcal{W}}^{cs}_{f}(x))\subset \tilde{\mathcal{W}}^{cs}_{A}(H_{f}(x))$. 

Once again, since $f$ is CF we have that both $H_{f}$ and $\Pi^{\sigma}\circ H_{f}$ are injective when restricted to $\tilde{\mathcal{W}}^{\sigma}_{f}$-leaves. Hence, $f$ verifies $(I^{\sigma})$ and also $(S^{\sigma})$ by remark \ref{iis}. Then by Lemma \ref{equipro} $f$ is $\sigma$-proper. 

To end the proof, we observed that $(S^{\sigma})$ and CF implies that 
$\tilde{\mathcal{W}}^{cs}_{A}(H_{f}(x))\subset H_{f}(\tilde{\mathcal{W}}^{cs}_{f}(x))$.
We conclude that 
$$ H(\tilde{\mathcal{W}}^{cs}_{f}(x))= \tilde{\mathcal{W}}^{cs}_{A}(H_{f}(x))$$
Since the center-unstable case is completely symmetric we obtain 
$$ H(\tilde{\mathcal{W}}^{cu}_{f}(x)) = \tilde{\mathcal{W}}^{cu}_{A}(H_{f}(x))$$
and from this we have SADC because $H_{f}$ is at bounded distance from the identity.  
\end{proof}

\end{document}